\documentclass[10pt,twocolumn,twoside]{IEEEtran}

\usepackage[margin=0.75in]{geometry}
\usepackage{amsmath,amssymb,amsthm}
\usepackage{psfrag}
\usepackage{cite}
\usepackage{graphicx}
\usepackage{graphics}
\usepackage{epstopdf}
\usepackage{float}
\usepackage{bbm}
\usepackage{framed}
\usepackage{textcomp}
\usepackage{subfigure}
\usepackage{enumerate}
\usepackage{accents}
\restylefloat{figure}
\usepackage{color,xspace}
\usepackage{tikz}
\usepackage{upgreek}
\usetikzlibrary{automata,positioning}

\usepackage{pstool}
\usepackage{url}

\renewcommand{\baselinestretch}{1}

\newcommand\munderbar[1]{%
  \underaccent{\bar}{#1}}

\newcommand{\real}{\mathbb{R}}
\newcommand{\realnonnegative}{\mathbb{R}_{\geq 0}}
\newcommand{\realpositive}{\mathbb{R}_{> 0}}
\newcommand{\complex}{\mathbb{C}}
\newcommand{\setdef}[2]{\{#1 \; | \; #2\}}

\newcommand{\virA}{\mathcal{A}}
\newcommand{\virB}{\mathcal{B}}
\newcommand{\A}{A} 
\newcommand{\Ain}{\A \text{in}}
\newcommand{\Bin}{\B \text{in}}
\newcommand{\B}{B} 
\newcommand{\bAstar}{\beta^{\virA \star}}
\newcommand{\G}{G} 
\newcommand{\GA}{G_A} 
\newcommand{\GB}{G_B} 

\newcommand{\N}{\mathcal{N}}

\newcommand{\V}{V} 
\newcommand{\VA}{\V^\A}

\newcommand{\VB}{\V^\B}

\newcommand{\E}{E} 
\newcommand{\EA}{\E^\A}
\newcommand{\EB}{\E^\B}

\usetikzlibrary{automata,positioning,fit,shapes.geometric}

\newcommand{\bAijstar}{{\beta}^{\virA \star}_{ij}}
\newcommand{\dAistar}{{\delta}^{\virA \star}_i}

\newcommand{\dAi}{\delta^{\virA}_i}
\newcommand{\dA}{\vec{\delta}^{\virA}}
\newcommand{\dB}{\vec{\delta}^{\virB}}
\newcommand{\dBi}{\delta^{\virB}_i}

\newcommand{\dmax}{\overline{\delta}}
\newcommand{\dmaxAi}{\hat{\delta}^\virA_i}

\newcommand{\hatbAij}{\hat{\beta}^\virA_{ij}}
\newcommand{\hatbAji}{\hat{\beta}^\virA_{ji}}
\newcommand{\hatbA}{\hat{\beta}^\virA}

\newcommand{\bAij}{\beta^{\virA}_{ij}}
\newcommand{\bAijlow}{\munderbar{\beta}^{\virA}_{ij}}
\newcommand{\bAijup}{\bar{\beta}^{\virA}_{ij}}
\newcommand{\dAilow}{\munderbar{\delta}^{\virA}_{i}}
\newcommand{\dAiup}{\bar{\delta}^{\virA}_{i}}
\newcommand{\tbAij}{\tilde{\beta}^{\virA}_{ij}}
\newcommand{\bAji}{\beta^{\virA}_{ji}}
\newcommand{\bA}{\beta^{\virA}}

\newcommand{\bBji}{\beta^{\virB}_{ji}}
\newcommand{\dPhiAi}{\dot{\Phi}^\virA_i}

\newcommand{\dPhiBi}{\dot{\Phi}^\virB_i}
\newcommand{\PhiA}{\vec{\Phi}^\virA}
\newcommand{\PhiAiss}{\bar{\Phi}^\virA_i}

\newcommand{\Phiss}{\bar{\Phi}}

\newcommand{\PhiB}{\vec{\Phi}^\virB}
\newcommand{\PhiBiss}{\bar{\Phi}^\virB_i}
\newcommand{\PhiBjss}{\bar{\Phi}^\virB_j}
\newcommand{\PhiAi}{\Phi^\virA_i}
\newcommand{\PhiAj}{\Phi^\virA_j}
\newcommand{\PhiBj}{\Phi^\virB_j}
\newcommand{\PhiBi}{\Phi^\virB_i}

\newcommand{\Posy}{\mathcal{P}}

\newcommand{\PsiA}{\vec{\Psi}^\virA}
\newcommand{\PsiB}{\vec{\Psi}^\virB}

\newcommand{\dotPsiA}{\dot{\Psi}^\virA}
\newcommand{\dotPsiB}{\dot{\Psi}^\virB}

\newcommand{\PhiBss}{\bar{\Phi}^\virB}
\newcommand{\PhiAss}{\bar{\Phi}^\virA}
\newcommand{\bB}{\beta^\virB}

\newcommand{\XiA}{X_i^{\virA}}
\newcommand{\XiB}{X_i^{\virB}}
\newcommand{\XiS}{X_i^{S}}

\newcommand{\XjA}{X_j^{\virA}}
\newcommand{\XjB}{X_j^{\virB}}

\newcommand{\IA}{I^{\virA}}

\newcommand{\IB}{I^{\virB}}

\newcommand{\budget}{\mathfrak{C}}

\newcommand{\jini}[1]{j \in \N^{#1}_i}
\newcommand{\goesto}{\rightarrow}

\newcommand{\lbr}{\left \{}
\newcommand{\rbr}{\right \}}
\newcommand{\diag}{\operatorname{diag}}

\newtheorem{lem}{\textbf{Lemma}}
\newtheorem{theorem}{\textbf{Theorem}}
\newtheorem{prop}{\textbf{Proposition}}
\newtheorem{cor}{\textbf{Corollary}}

\newcommand{\oprocendsymbol}{\hbox{$\bullet$}}
\newcommand{\oprocend}{\relax\ifmmode\else\unskip\hfill\fi\oprocendsymbol}

\pgfdeclarelayer{bg}
\pgfdeclarelayer{bbg}
\pgfsetlayers{bbg,bg,main}

\newtheorem{problem}{\textbf{Problem}}{}
\newtheorem{remark}{\textbf{Remark}}{}

\begin{document}
\title{Optimal resource allocation for competitive spreading processes on bilayer networks}
\author{Nicholas J. Watkins, Cameron Nowzari,  Victor M. Preciado, and George J. Pappas\thanks{The authors are with the Department of Electrical and Systems Engineering, University of Pennsylvania, Pennsylvania, PA 19104, USA, {\tt\small \{nwatk,cnowzari,preciado,pappasg\}@upenn.edu}}}
      
\maketitle
\begin{abstract}
This paper studies the $S I_1 S I_2 S$ spreading model of two competing behaviors over a bilayer network.  We address the problem of determining resource allocation strategies which design a spreading network so as to ensure the extinction of a selected process.  Our discussion begins by extending the $S I_1 S I_2 S$ model to edge-dependent infection and node-dependent recovery parameters with generalized graph topologies, which builds upon prior work that studies the homogeneous case.  We then find conditions under which the  mean-field approximation of a chosen epidemic process stabilizes to extinction exponentially quickly. Leveraging this result, 
we formulate and solve an optimal resource allocation problem in which we minimize the expenditure necessary to force a chosen epidemic process to become extinct as quickly as possible.
In the case that the budget is not sufficient to ensure extinction of the desired process,
we instead minimize a useful heuristic.  We explore the efficacy of our methods by comparing simulations of the stochastic process to the mean-field model, and find that the mean-field methods developed work well for the optimal cost networks designed, but suffer from inaccuracy in other situations.
\end{abstract}

\section{Introduction} \label{sec:Intro}
Modeling, analysis, and control of spreading processes in complex networks has recently garnered significant attention from the research community.  The potential applications for such methods are diverse: the spread of biological epidemics, social behaviors, and cybersecurity threats can all be formalized within this framework.  Prior work has focused primarily on the case of single-layer spreading networks, however it is clear that such an abstraction is limited in modeling capacity.  
In principle, many real world networks transmit phenomena through markedly different channels, which motivates the study of multi-layer models such as the one addressed here.

This paper studies a multi-layer, heterogeneous compartmental epidemic model, in which the spread of competing beliefs and behaviors through social interactions can be modeled.  We direct our attention to a set of problems focused on a single theme - that of controlling a spreading process so as to quickly eliminate a chosen epidemic while allowing the possibility that the other survives indefinitely.  This is a natural choice of equilibrium concept for several socially relevant problems.  For example, we may use this model to study the effects of political strategies on the opinions of the populace, predict the ramifications of gossip in professional networks, and understand the influence of marketing strategies on consumer behavior.





\paragraph*{Literature review}
Many well-known models of spreading processes in networks are developed for the case of a single contagion spreading over a single network layer; we refer the reader to \cite{Hethcote_Math09,Keeling_Networks05,nowzari2015analysis} for an overview. 
Recent efforts have been made in extending this body of work to account for  the possibility of competitive and\slash or coexistent processes on single-layer networks.  Particular examples include investigations into the effects of multiple pathogens in a single-layer `Susceptible-Infected-Removed' ($SIR$) model \cite{Newman_Thresh05,Karrer_Competing11,Newman_Interacting13}, a study of an extension to the $SIR$ model ($SICR$) for assessing the effects of competition and cooperation between pathogens spreading on a single network \cite{Shrestha_StatInfer11}, and the development of a model for the spread of competing ideas using the `Susceptible-Infected-Susceptible' ($SIS$) model on scale-free networks \cite{Wang_DynamicsComp11}.

A more recent trend is the investigation into systems with multiple pathogens \emph{and} multiple spreading layers. An overview of this research area can be found in \cite{Salehi_DiffProc14}. Particular examples of interest include an investigation into the effects of pathogen interaction on overlay networks with $SIR$ dynamics \cite{Funk_Interacting10}, the development of a model in which disease awareness and infection spread on separate layers of $SIS$ dynamics \cite{Granell_DynInter13,Granell_Competing14}, the development of a model ($SI_1 S I_2S$) that generalizes the classic $SIS$ model to a competitive multilayer framework \cite{Wei_CompetingMemes13}, and work to find conditions under which processes in the $SI_1 S I_2S$ model can coexist \cite{Sahneh_CompSpread14}.



We concern ourselves with finding resource allocations which design a network to control the system at optimal cost.  Similar problems have been studied for controlling the mean-field approximation for the single layer $SIS$ model in \cite{Preciado_OptRecAlloc14}, and a non-competitive multilayer model in \cite{Chen_CoInfect14}.  The work we present here is the first to consider an allocation problem which leverages inter-process competition, which we incorporate by studying a variant of the $S I_1 S I_2 S$ process.


\paragraph*{Statement of contributions}
Our primary contribution is in developing resource allocation method which designs a network wherein the mean-field approximation of the $S I_1 S I_2 S$ process presented in \cite{Wei_CompetingMemes13} and \cite{Sahneh_CompSpread14} is controlled to a desired equilibrium.  To accomplish this, we introduce a more general, heterogeneous version of the model with arbitrary spreading topologies, so as to enable us to capture the effects of asymmetric influence among the agents, which we then leverage in the design of networks which exponentially eliminate one epidemic while allowing the possibility that the other survives in an endemic state. We believe this equilibrium concept is useful in applications of various competitive spreading problems, such as a marketing firm wanting to influence their customer base to purchase a certain product over others.  Our technical contributions evolve from addressing this task, and address several control theoretic facets of the mean-field $S I_1 S I_2 S$ model left previously unexplored.

More specifically, we first determine necessary and sufficient conditions for exponentially stabilizing the desired equilibrium of the mean-field model.  Then, we formulate an 
optimal resource allocation problem where we may pay specified costs in order to assign particular values to model parameters.  We develop tractable methods for computing a minimal-cost set of resource allocations which attains the desired equilibrium, and for mitigating the prevalence of the unwanted epidemic process when the available budget is not sufficient for realizing
the desired equilibrium.  Finally,
we explore the efficacy of the mean-field control policies developed against the stochastic process behavior through extensive Monte Carlo simulations.  With respect to the preliminary work presented in \cite{watkins2015optimal}, this paper extends the results pertaining to the effects of competition, provides proofs of our main results, and adds significant simulations comparing the mean-field model to the stochastic $S I_1 S I_2 S$ process.

\subsection{Notation and Mathematical Review}
\paragraph*{Vectors and Matrices}
Let $\real$, $\realnonnegative$, and $\realpositive$ denote the set of real, nonnegative real numbers, and positive real numbers respectively. 
We use the notation $\vec{x} \in \real^n$ to denote an $n$-dimensional column vector, and $\vec{x}^T$ to denote its transpose, both with components $x_i \in \real$.  We use $|S|$ to denote the cardinality of a finite set.

We say a matrix $A$ is irreducible if no similarity transformation exists which places $A$ into block upper-triangular form.  We denote by $\diag(\vec{a})$ a matrix with entries $\diag(\vec{a})_{ii} = a_i$ for all $i$ and $0$ elsewhere.  We will make repeated use of the Perron-Frobenius Theorem, which gives:
\begin{prop}[P-F Theorem] \label{prop:pf}
	Let $A$ be a non-negative, irreducible matrix.  Then, there exists a vector $\vec{u}$ such that $u_i > 0$ for all $i$, and $A u = \lambda^* u$, where $\lambda^*$ is the leading eigenvalue of $A$.
\end{prop}

\paragraph*{Graph Theory}
A \emph{directed graph} (digraph) is given by a triplet $\G = \lbr \V, \E, \A \rbr$ in which $\V$ is the set of vertices, $\E \subseteq \V \times \V$ the ordered set of edges, and $\A \in \{0,1\}^{|V| \times |V|}$ is the adjacency matrix.  In such a graph, $a_{ij} = 1$ if and only if there exists an edge $(i,j) \in \E$ connecting node $i$ to node $j$. We define
the set of in-neighbors of node $i$ given the adjacency matrix $\A$ as
$\N^{\A \text{in}}_i = \setdef{j \in \V}{a_{ji} = 1}$. 

A path $p$ is given by an ordered set of vertices $p = \lbr v_1, v_2, \dots, v_m \rbr$ such that for each pair of consecutive vertices $v_k$, $v_{k+1}$, $\left(v_k, v_{k+1} \right)$ is an edge in $E$.  We say that some path $p$ connects node $v_i$ and $v_j$ if both $v_i$ and $v_j$ are listed as nodes in the path.  We say a digraph is strongly connected if there exists some path $p$ connecting node $v_i$ to node $v_j$ for all $v_i, v_j \in \V$. The adjacency matrix of a strongly-connected digraph is irreducible.

A \emph{bilayer graph} is a collection of two graphs, $\G = \lbr \GA, \GB \rbr$ which satisfy the following property: the vertex set $\V$ and edge set $\E$ of $\G$ are such that $\V = \VA \,\cup \, \VB$, and $\E = \EA \, \cup \, \EB$, where $\VA$ and $\VB$ are the vertex sets of $\GA$ and $\GB$, respectively, and $\EA$ and $\EB$ are the edge sets of $\GA$ and $\GB$, respectively.  

\paragraph*{Geometric Programming}
Geometric programs form a class of quasiconvex optimization problems which have \emph{posynomial} objective functions and inequality constraints, and \emph{monomial} equality constraints.
A function $f : \realpositive^n \rightarrow \real$ is called a \emph{monomial} if it can be written in the form $f(\vec{x}) = c \, x_1^{r_1} x_2^{r_2}\dots x_n^{r_n}$, where $c > 0$ is used to denote a leading constant, the $r_i$ terms represent constant powers to which the arguments are raised, and the $x_i$ terms represent $f$'s arguments.  A function is said to be a \emph{posynomial} if it can be written as a sum of monomials.

Geometric programs can be made into convex optimization problems by performing a logarithmic change of variables and a logarithmic transformation of the objective and constraint functions.  For further details on geometric programs and their solution, we refer the reader to \cite{Boyd_CvxOpt04,Boyd_GPTut07}.

\section{Model and Problem Statement} \label{sec:Problem}
We begin our technical discussion by extending the $SI_1SI_2S$ model proposed in \cite{Wei_CompetingMemes13} and analyzed further in \cite{Sahneh_CompSpread14}.  Our primary contribution in extending this model is allowing the processes to be influenced by heterogeneous parameters, and allowing for the graph layers to be strongly connected digraphs with arbitrary node sets.  This contrasts with the work in \cite{Sahneh_CompSpread14}, which assumes homogeneous spreading parameters and undirected layers with identical node sets.  The extension allows the possibility of modeling asymmetric influence and nodal immunity, and is required to formulate the resource allocation problem.

We consider the spread of epidemics $\virA$ and $\virB$ over a bilayer graph $G = \{\GA,\GB\}$, where $\virA$ spreads over $\GA = \{\VA,\EA, A\}$, $\virB$ spreads over $\GB = \{\VB,\EB,B\}$, and $|V| = n$. At any time $t$, we assume that each node can belong to one of three \emph{compartments}: $\IA$ if the the node is infected by epidemic $\virA$, $\IB$ if the node is infected by epidemic $\virB$ and $S$ if the node is infected by neither.  We let $\XiA$, $\XiB$, and $\XiS$ 
denote indicator functions corresponding to the compartments $\IA$, $\IB$ and $S$, respectively,  where we define $\XiA(t) = 1$ if node $i$ is in compartment $\IA$ at time $t$ and $\XiA(t) = 0$ otherwise.  We define $\XiB$ and $\XiS$ similarly.

We model the spread of $\virA$ and $\virB$ as a Markovian contact process in which a node $i$ in compartment $S$ transitions to $\IA$ whenever it is a contacted by a node $j$ in compartment $\IA$,
with similar considerations holding for transitions from $S$ to $\IB$.  We assume all of the contact processes are stochastically independent, and occur at rates $\bAji$ for the transitions from $S$ to $\IA$ and $\bBji$ for the transitions from $S$ to $\IB$, which we refer to as \emph{spreading} rates.  From this description, it then follows that the process which transitions node $i$ from compartment $S$ to compartment $\IA$ is a Poisson process with rate $Y_i^{\mathcal{A}}(t) = \sum_{\jini{\Ain}} \bAji \XjA(t),$ and the process which transitions node $i$ from compartment $S$ to compartment $\IB$ is a Poisson process with rate $Y_i^{\mathcal{B}}(t) = \sum_{\jini{\Bin}} \bBji \XjB(t).$  The processes which transition a node $i$ from $\IA$ to $S$ and from $\IB$ to $S$ are Poisson processes with rates $\dAi$ and $\dBi$, which we refer to as \emph{healing} rates.   An illustration of the heterogeneous $S I_1 S I_2 S$ process is given in Figure \ref{fig:compartmental}.  

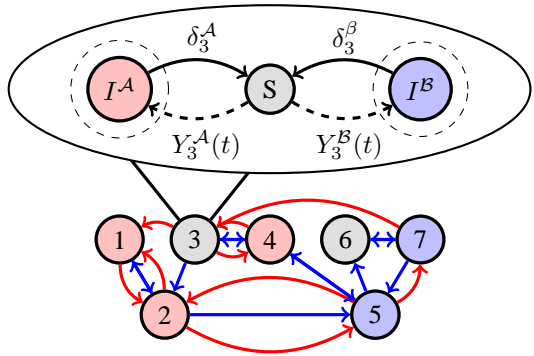
\begin{figure}[h!]
	\centering
	\begin{tikzpicture}
	\node[circle,very thick,draw=black] (IA) at (-2,0) {$I^\virA$};
	\node[circle,very thick,draw=black] (IB) at (2,0) {$I^\virB$};
	\node[circle,very thick,draw=black] (S) at (0,0) {S};
	\node[draw=black,dashed,fit=(IA) ,inner sep=0.1ex,ellipse] (I1circ) {};
	\node[draw=black,dashed,fit=(IB) ,inner sep=0.1ex,ellipse] (I2circ) {};
		
	\path [->,very thick,dashed] (S) edge [bend left] node[xshift=0.1cm,yshift=-0.35cm] {$Y_3^{\mathcal{A}}(t)$} (IA);
	\path[->,very thick] (IA) edge [bend left] node[xshift=0.05cm, yshift=0.3cm] {$\delta^\mathcal{A}_3$} (S);
	\path [->,very thick,dashed] (S) edge [bend right] node[xshift=0.1cm,yshift=-0.35cm] {$Y_3^{\mathcal{B}}(t)$} (IB);
	\path [->,very thick] (IB) edge [bend right] node[xshift=0.05cm, yshift=0.3cm] 
	{$\delta^{\beta}_3$} (S);
	
	\node[circle,very thick,draw=black] (1) at (-2,-2) {1};
	\node[circle,very thick,draw=black] (2) at (-1.4,-3) {2};
	\node[circle,very thick,draw=black] (3) at (2,-2) {7};
	\node[circle,very thick,draw=black] (4) at (1.4,-3) {5};
	\node[circle,very thick,draw=black] (5) at (1,-2) {6};
	\node[circle,very thick,draw=black] (6) at (-1,-2) {3};
	\node[circle,very thick,draw=black] (7) at (0,-2) {4};
	
	\path [->,very thick,color=red] (1) edge [bend right] (2);
	\path [->,very thick,color=red] (2) edge [bend right] (1);
	\path [->,very thick,color=blue] (2) edge (1);
	\path [->,very thick,color=blue] (1) edge (2);
	
	\path [->,very thick,color=red] (2) edge [bend right] (4);
	\path [->,very thick,color=red] (4) edge [bend right] (2);
	\path [->,very thick,color=red] (4) edge [bend right] (3);
	\path [->,very thick,color=blue] (3) edge (4);
	\path [->,very thick,color=blue] (4) edge (5);
	
	\path [->,very thick,color=red] (3) edge [bend right] (6);
	\path [->,very thick,color=red] (6) edge [bend right] (7);
	\path [->,very thick,color=red] (7) edge [bend right] (6);
	\path [->,very thick,color=red] (6) edge [bend right] (1);
	\path [->,very thick,color=blue] (3) edge (5);
	\path [->,very thick,color=blue] (5) edge (3);
	\path [->,very thick,color=blue] (4) edge (7);
	\path [->,very thick,color=blue] (7) edge (4);
	\path [->,very thick,color=blue] (7) edge (6);
	\path [->,very thick,color=blue] (6) edge (7);
	\path [->,very thick,color=blue] (6) edge (2);
	\path [->,very thick,color=blue] (2) edge (4);
	
	\node[] (anch1) at (0,.65) {};
	\node[] (anch2) at (0,-.65) {};
	\node[] (ell1) at (0,-.75) {};
	\node[] (ell2) at (-2,-.75) {};
	
	\begin{pgfonlayer}{bg}
	\node[draw=black,fit={(IB) (S) (IA) (anch1) (anch2)},thick ,inner sep=0.1ex,ellipse,fill=white] (I2circ) {};
	\node[circle,very thick,draw=black,opacity=0.25,fill=red] (IA) at (-2,0) {$I^\virA$};
	\node[circle,very thick,draw=black,opacity=0.25,fill=blue] (IB) at (2,0) {$I^\virB$};
	\node[circle,very thick,draw=black,opacity=0.25,fill=gray] (S) at (0,0) {S};
	
	\node[circle,very thick,draw=black,opacity=0.25,fill=red] (1) at (-2,-2) {1};
	\node[circle,very thick,draw=black,opacity=0.25,fill=red] (2) at (-1.4,-3) {2};
	\node[circle,very thick,draw=black,opacity=0.25,fill=blue] (3) at (2,-2) {7};
	\node[circle,very thick,draw=black,opacity=0.25,fill=blue] (4) at (1.4,-3) {5};
	\node[circle,very thick,draw=black,opacity=0.25,fill=gray] (5) at (1,-2) {6};
	\node[circle,very thick,draw=black,opacity=0.25,fill=gray] (6) at (-1,-2) {3};
	\node[circle,very thick,draw=black,opacity=0.25,fill=red] (7) at (0,-2) {4};
	\end{pgfonlayer}
	
	\begin{pgfonlayer}{bbg}
	\path [very thick,] (6) edge (ell1);
	\path [very thick] (6) edge (ell2);
	\end{pgfonlayer}
	\end{tikzpicture}
	\caption{A diagram of the $S I_1 S I_2 S$ process, with the spreading graph for $\A$ given by red edges, and the spreading graph for $\B$ given by blue edges.  The transition process for node $3$ is explicitly illustrated, where we note that node $3$ is a member of both spreading graphs, and so may have transitions to both $\IA$ and $\IB$.} \label{fig:compartmental}
\end{figure}


For a general instance of the $S I_1 S I_2 S$ process, studying the exact dynamics would require the enumeration of a Markov process with $O(3^n)$ states, 
arising from the need to explicitly account for all permissible combinations of compartmental memberships allowed by the instance of the problem.  There are at least two methods of dealing with this complexity: (i) restricting considerations to simple graph topologies, and (ii) approximating the dynamics by a lower-dimensional system.  As our goal is to design resource allocations on graphs with arbitrary graph structures, we consider here a mean-field approximation of the process, which reduces the dimension of the system's state space to $O(2n)$.

To clearly demonstrate how we arrive at the mean-field dynamics and give insight as to what effects the enacted approximations make, we first consider the exact equations of the process dynamics:
\begin{equation} \label{eq:exp_dynamics}
\begin{aligned}
	&\frac{d\E[\XiA]}{d t} = \E\left[(1 - \XiA - \XiB)\sum_{\jini{\Ain}} \bAji \XjA - \dAi \XiA \right] \\
	&\frac{d \E[\XiB]}{d t} = \E \left[(1 - \XiA - \XiB)\sum_{\jini{\Bin}} \bBji \XjB - \dBi \XiB \right]\\
\end{aligned}
\end{equation}
where we have used the substitution $\XiS = (1 - \XiA - \XiB)$ in order to reduce dimension.  Note that the equations described by the system \eqref{eq:exp_dynamics} are not closed: they contain terms of the form $\E[\XiA\XjA]$ and $\E[\XiB\XjA]$, which cannot be represented in terms of the dynamics of $\E[\XiA]$ and $\E[\XiB]$ without incurring error.  However, without a closed set of equations we cannot perform analysis, and so we make the approximations $\E[\XiA\XjA] \approx \PhiAi \PhiAj$ and $\E[\XiB\XjA] \approx \PhiBi \PhiAj$, where we have introduced the symbols $\PhiAi$ and $\PhiBi$ to denote the mean-field states approximating the probability that node $i$ is in $\IA$, and the probability that node $i$ is in $\IB$, respectively.
This substitution allows us to arrive at a mean-field approximation of $S I_1 S I_2 S$ in the style of \cite{Sahneh_GEMF13}:
\begin{align}
	&\dPhiAi = (1-\PhiAi-\PhiBi)\sum_{\jini{\Ain}} \bAji \PhiAj - \dAi \PhiAi , \label{eq:dynA}\\
	&\dPhiBi = (1-\PhiAi-\PhiBi)\sum_{\jini{\Bin}} \bBji \PhiBj - \dBi \PhiBi . \label{eq:dynB}
\end{align}
We will more thoroughly examine the interrelation of the mean-field model and the stochastic process in Section \ref{sec:Sim}.  However, the majority of our work will be guided by seeking answers to the following questions with respect to the mean-field model:
\begin{enumerate}[(a)]
	\item Extinction: what conditions are sufficient to extinct a chosen process quickly?
	\item Optimal Extinction: can we compute an optimal allocation of resources to attain a desired extinction quickly?
	\item Fixed Budget Mitigation: given a fixed budget, can we limit the prevalence a desired process effectively?
\end{enumerate}

We believe answers to these questions are
of interest to the community of researchers currently engaged in the study of competitive epidemic spreading processes.  As a particular example of a future application, we may consider a situation in which a firm would like to quell negative word-of-mouth advertising on its network of customers in the most expedient and cost effective manner possible.  We may represent this within the framework of our model as a problem of finding conditions under which an unwanted process is driven out of existence as quickly and efficiently as possible. 
Our work shows that computing an optimal-cost network to realize this goal is feasible in the mean-field regime, and provides a step forward from the earlier works considering single-layer spreading processes.


\section{Extinction Conditions} \label{sec:extinct}
This section addresses the first of our stated problems, i.e. finding conditions under which the unwanted epidemic extincts, or more concretely:
\begin{problem}[Extinction] \label{prob:stabilization}
	For some specified $S I_1 S I_2 S$ spreading process on a bilayer graph $G$, determine conditions for the parameters of the subgraph $\GA$ under which a chosen behavior $\virA$ extincts quickly, while allowing the possibility the behavior $\virB$ survives indefinitely.
\end{problem}

In particular, we are concerned with stabilizing a mean-field equilibrium $\Phiss = [(\PhiAss)^T, (\PhiBss)^T]^T$
where we have that $\PhiAiss$ and $\PhiBiss$ are the steady states of $\PhiAi$ and $\PhiBi$, $\PhiAiss = 0$ for all $i$, and the values of $\PhiBiss$ are given by the solutions of the system:
\begin{align}
	\frac{\PhiBiss}{(1-\PhiBiss)} = \frac{1}{\dBi} \sum_{\jini{\Bin}} \bBji \PhiBjss, \label{eq:equilB}
\end{align}
which may be computed numerically by methods similar to those used in \cite{VanMieghem_Virus09} for the $SIS$ steady-state equations, and is unique due to the uniqueness of the $SIS$ endemic equilibrium \cite{khanafer2014stability}.  With the ability to claim knowledge of the values $\{\PhiBiss|_{\PhiAiss = 0} \}_{i \in V}$, we may now construct a result to Problem \ref{prob:stabilization}.  In fact, we find necessary and sufficient conditions for the desired equilibrium to be exponentially stable:
\begin{theorem}[Mean-field Exponential Stability] \label{thm:stab}
For any $S I_1 S I_2 S$ spreading process on a strongly connected bilayer graph $G$ with mean field dynamics given
by~\eqref{eq:dynA} and~\eqref{eq:dynB}, the equilibrium $\Phiss = \left[(\PhiAss)^T, (\PhiBss)^T \right]^T$ with $\PhiAiss = 0$ for all $i$ and $\PhiBss$ given by the solutions of \eqref{eq:equilB} is (locally) exponentially stable if and only if
\begin{align*}
J_{11} = \diag \left(1 - \PhiBss\right) (\bA)^T  - \diag \left( \dA \right)
\end{align*}
is Hurwitz, where $\dA$ is the vector of $\virA$'s recovery rates, and $\bA$ is the matrix of $\virA$'s spreading rates, which we assume to inherit $A$'s sparsity pattern.
\end{theorem}
\vspace{-0.3cm}
\begin{proof}
We begin by computing the linearization of the mean-field dynamics given by \eqref{eq:dynA} and \eqref{eq:dynB} about $\Phiss$, which we can show to be:
\begin{equation}\label{eq:linear}
	\left[\begin{matrix}
		\dotPsiA \\ 
		\dotPsiB
	\end{matrix} \right] =
	\left(\begin{matrix} 
 		J_{11} & 0 \\
		J_{21} & J_{22}
	\end{matrix} \right)
	\left[\begin{matrix}
		\PsiA \\
		\PsiB
	\end{matrix} \right] = 
	J \left[\begin{matrix}
			\PsiA \\
			\PsiB
		\end{matrix} \right] ,
\end{equation}
where
\begin{align*}
	J_{11} =& \diag \left(1-\PhiBss \right) (\bA)^T  - \diag(\dA) , \\
	J_{21} =& -\diag\left((\bB)^T \, \PhiBss\right) , \\
	J_{22} =& \diag  \left(1-\PhiBss \right) (\bB)^T  \\
	&- \diag((\bB)^T \, \PhiBss + \dB),
\end{align*}
with $\dB$ and $\bB$ defined analogously to $\dA$ and $\bA$.

We note also that the constituent matrices are constant, hence they are componentwise bounded and Lipschitz, and thereby allow the application of a well-known result from the theory of nonlinear systems:

\begin{prop} [Exp. Stability \cite{Khalil_Nonlinear02}] \label{prop:hurwitz}
	Let $x_0$ be an equilibrium point of the nonlinear system $x = f \left( x \right)$, where $f : D \rightarrow \real^n$ is continuously differentiable and the Jacobian matrix is bounded and Lipschitz on D.  Let 
	$$M = \dfrac{\partial f}{\partial x} \Big\vert_{x = x_0} $$
	Then, $x_0$ is an exponentially stable equilibrium point for the nonlinear system if and only if it is an exponentially stable equilibrium point of the linear system $\dot{x} = M x$. 
\end{prop}

Since the Jacobian matrix J of the system is bounded and Lipshitz, we have that the nonlinear dynamics given by \eqref{eq:dynA} and \eqref{eq:dynB} are (locally) exponentially stable if and only if the linearized system~\eqref{eq:linear} is exponentially stable.  It remains to show that the hypothesis ensures that $J$ is Hurwitz.

We note that, due to the block $0$ in the upper-right entry of $J$, the eigenvalues of $J$ are given by the eigenvalues of the matrices $J_{11}$ and $J_{22}$.  Noting that $J_{11}$ is indeed the matrix the hypothesis claims to be Hurwitz, we may turn our attention to $J_{22}$.

The $J_{22}$ matrix is exactly the Jacobian of the dynamics of a single-virus, single-layer n-Intertwined $SIS$ system evaluated at its metastable equilibrium.  We may now call upon a prior result from the literature of single-layer $SIS$ processes to complete our analysis:
\begin{prop}[Single-layer Exp. Stability \cite{khanafer2014stability}] \label{prop:khanafer}
	Suppose $p^*$ is a equilibrium of the n-Intertwined $SIS$ dynamics occurring on an arbitrary single-layer graph.  Then $p^*$ is globally asymptotically stable, and locally exponentially stable.
\end{prop}

We may now use Proposition \ref{prop:khanafer} to claim that $\PhiBss$ is a locally exponentially stable equilibrium point of the simplified (i.e. single-layer) model.  By Proposition \ref{prop:hurwitz} it must be that $J_{22}$ is Hurwitz, as it is componentwise bounded and Lipshitz.  

Since both $J_{11}$ and $J_{22}$ are Hurwitz, it must be that $J$ is Hurwitz.  Hence, it must be that $\Phiss$ is an exponentially stable equilibrium  point of the nonlinear system described by \eqref{eq:dynA}-\eqref{eq:dynB}.  Since all of the relations used in the proof are equivalences, no further considerations are necessary. 
\end{proof}


\begin{remark}[Homogeneous Threshold]
{\rm
	Note that this is similar to, but more general than, the stability results presented in \cite{Sahneh_CompSpread14}.  In particular, the condition in \cite{Sahneh_CompSpread14} requires that all infection rates $\bAij$ and recovery rates $\dAi$ take on homogeneous values $\beta$ and $\delta$ such that
	$\frac{\beta}{\delta} < \frac{1}{\lambda_{\max}(\diag(1 - \PhiBss)A)}.$  
	By inspection, it is clear that our result permits parameter choices which are excluded by this condition.  
	} \oprocend
\end{remark}

The form of the matrix we need to stabilize to guarantee the extinction of epidemic $\virA$ is similar to the matrix needed to guarantee extinction when we ignore competition.  In particular, we note that a simple consequence of prior work on the $SIS$ process (see, e.g., \cite{Preciado_OptRecAlloc14}) is that a sufficient condition for the exponentially fast elimination of the process spreading $\virA$ is that $\lambda_{\max}((\bA)^T  - \diag(\dA)) < 0$ holds.  By accounting for persistent competition among the epidemic processes, we might expect that our condition allows for more aggressive parameter selections.  We will show that this is true in a rigorous sense with our next result, which we will develop by first considering a technical lemma, and then specializing to our setting.

\begin{lem}[Row Compression Inequality] \label{lem:row_comp}
Let $M \in \realnonnegative^{(n \times n)}$, $\vec{\alpha} \in [0,1]^n$ and $\vec{\gamma} \in \real^n$.  Then, the following inequality holds:
	\begin{equation}
	\lambda_{\max} \left( \diag(\vec{\alpha})M - \diag(\vec{\gamma}) \right) \leq \lambda_{\max} \left(M - \diag(\vec{\gamma}) \right).
	\end{equation}
\end{lem}
\begin{proof}
	See Appendix.
\end{proof}

Now, we may see an immediate consequence of Lemma \ref{lem:row_comp} - the presence of competition in any particular node helps to prevent the persistence of an unwanted behavior.
We make this formal as follows:
 
\begin{cor}[Benefits of Competition] \label{thm:competition}
	Take any set of values $(\bA,\dA)$ such that $\lambda_{\max}((\bA)^T  - \diag(\dA)) < 0$ holds and $\dAi > 0$ for all $i$.  Then, for any instance of the $S I_1 S I_2 S$ process, we must also have exponential elimination of $\virA$.
	
	Moreover, if $\PhiBss$ is such that $\PhiBiss > 0$ for some $i$, then there exists some $\hatbA$ with $\hatbAij \geq \bAij$ for all $i$ and $j$, $\hatbAij > \bAij$ for some $i$ and $j$, and $(\hatbA,\dA)$ exponentially eliminates $\virA$.
\end{cor}
\begin{proof}
	The first claim is a direct consequence of Theorem \ref{thm:stab} when we apply Lemma \ref{lem:row_comp} with $M = (\bA)^T$, $\vec{\alpha} = (1 - \PhiBss)$ and $\vec{\gamma} = \dA$.
	To prove the second claim, consider the matrix $\hatbA$ with entries $\hatbAji = \frac{1}{(1- \PhiBiss)} \bAji$.  Then, $$\diag(1-\PhiBss) (\hatbA)^T = (\bA)^T.$$  From our hypothesis that $\dAi > 0$ for all $i$, we have that $\PhiBiss \in [0,1)$ for all $i$.  Hence, $\hatbAij \geq \bAij$ for all $i$ and $j$, where the inequality is strict for the case where $\PhiBiss > 0$. 
\end{proof}

It is interesting to note that the result of Corollary \ref{thm:competition} admits an explicit characterization of \emph{how much} competition helps: for all agents $i$, we may scale the spreading rates associated to the incoming edges of node $i$ by a factor of at least $\frac{1}{(1-\PhiBiss)}$.  While the utility of this observation depends on the particular cost functions encountered in a given problem instance, this result does provide a clear benefit to considering competitive effects when considering resource allocations in a competitive environment. 

\section{Optimal Resource Allocations} \label{sec:opt_stab}

Having established conditions for exponential stability of the desired equilibrium, we now focus our attention on establishing means for designing resource allocations which create networks with desirable control properties.  We first consider the problem of designing a set of resource allocations so as to eliminate a chosen process at optimal cost when we are given functions which relate the chosen process parameter values to resource expenditures.  

In the context of a marketing problem, we may think of spending on resources such as product giveaways, consumer incentive programs, advertisement campaigns, etc. in order to affect the perception of a company within a given market.  To model this effect, we assume that for every designable parameter $\bAij$ and $\dAi$, we are given cost functions $f_{ij}$ and $g_i$ which relate a desired parameter value to a capital expenditure, the particular characteristics of which we assume to be application-specific.  With this notion developed, we may state our problem more formally as follows:
\begin{problem}[Optimal Extinction]\label{prob:min_stabilization}
	For some specified $S I_1 S I_2 S$ spreading process on a bilayer graph $G$ and given sets of cost functions $\{f_{ij}\}_{(i,j) \in \EA}$, $\{g_i\}_{i \in \VA}$, determine a minimum cost allocation of resources to enforce the extinction conditions for the equilibrium of Problem \ref{prob:stabilization}.
\end{problem}

From the discussion in Section \ref{sec:extinct}, we may formally cast Problem \ref{prob:min_stabilization} as the optimization program:
\begin{equation} \label{prog:opt_extinct}
	\begin{aligned}
	& \underset{\{\bA,\vec{\delta}^\virA\}}{\text{minimize}}
	& & \sum_{\{(i,j) \in E^A\}} f_{ij} \left(\bAij \right) + \sum_{\{i \in V^A\}} g_i \left(\dAi \right) \\
	& \text{subject to}
	& & \lambda_{\max}(J(\bA, \vec{\delta}^\virA) ) < 0\\
	\end{aligned}
\end{equation}
where $J$ is defined as in Theorem \ref{thm:stab}.  Note that  \eqref{prog:opt_extinct} is non-convex in general; it is an eigenvalue problem.  However, if we allow ourselves to restrict considerations to a reasonable class of cost functions, we may extend the work in \cite{Preciado_OptRecAlloc14} to arrive at a tractable solution.  

In particular, we will consider a method for transforming \eqref{prog:opt_extinct} into a convex problem when the cost functions are structured to make aggressive processes - those with higher spreading rates - more costly.  To ease the formal statement of the result, we will first introduce the notion of a posynomial transformation:
\begin{lem}[Posynomial Transformations] \label{lem:posy}
	Any function $f(x)$ of the form $f(x) = \sum_{k} c_k (\hat{x} - x)^{p_k}$ 
	with domain $\left(0, \hat{x} \right)$ with $\hat{x} > 0$, $c_k > 0$, and $p_k \in \real$ 
can be written as a posynomial function of a new variable $z = \hat{x} - x$ defined on the domain $\left(0, \hat{x} \right)$.
\end{lem}
\begin{proof}
	Consider the variable substitution $z = \hat{x} - x$.  Then, we may write the posynomial transformation $\hat{f} \left(\hat{x} - x \right) = \sum_{k} c_k (z)^{p_k},$
	where we see that a value $z = 0 \mapsto x = \hat{x}$ and $z = \hat{x} \mapsto x = 0$.  Since the transformation is continuous, the domain of $\hat{f}$ is $\left(0, \hat{x} \right)$, as specified by the hypothesis.
\end{proof}

We will denote the class of functions with domain $(0,d)$ which admit a posynomial transformation in the sense of Lemma \ref{lem:posy} by $\Posy(0,d)$.
This class of functions will appear repeatedly in the remainder, and will see its first use in the following result:

\begin{theorem} \label{thm:dyn_stable}
	Consider an instance of the dynamics \eqref{eq:dynA}-\eqref{eq:dynB} with an equilibrium point of the form $\Phiss = \left[(\PhiAss)^T, (\PhiBss)^T \right]^T$ with $\PhiAiss = 0$ for all $i$ and $\PhiBss$ given by the solutions of \eqref{eq:equilB}.  Define $z_i = \left(1 - \PhiBiss \right)$ for all $i$. 
	
	Then, for any $S I_1 S I_2 S$ spreading process on a strongly connected bilayer graph $G$, any set of monotonically decreasing posynomial cost functions $\lbr f_{ij} \rbr_{\{(i,j) \in E^A\}}$, any set of functions $\{g_i \in \Posy(0,\dmaxAi)\}_{i = 1}^{|V^A|},$ 
	and any $\epsilon \in \left(0, \underset{i}{\min} \lbr \dmaxAi \rbr \right)$, an optimal solution of \eqref{prog:opt_extinct} can be computed by the solution of the following geometric program:
	\begin{equation} \label{prog:thm2}
		\begin{aligned}
			& \underset{\{\bA, \vec{t}, \lambda, \vec{u} \}}{\text{minimize}}
			& & \sum_{\{(i,j) \in E^A\}} f_{ij} \left(\bAij \right) + \sum_{\{i \in V^A\}} \hat{g}_i \left(t_i \right) \\
			& \text{subject to}
			& & \frac{\sum_{\jini{\Ain}} \bAji z_i u_j + t_i u_i + \epsilon u_i}{\lambda u_i} \leq 1 \; \forall i,\\
			&&& \frac{t_i}{\dmax} \leq 1 \; \forall i,\\
			&&& \frac{\left(\dmax - \dmaxAi \right)}{t_i} \leq 1 \; \forall i, \\
			&&& \bAij, \, u_i  \geq 0 \; \forall i,j , \\
			&&& 0 \leq \lambda \leq \dmax
		\end{aligned}
	\end{equation}
	where $\dmax > \underset{i}{\max} \lbr \dmaxAi \rbr $, $\hat{g}_i$ denotes the posynomial transform of $g_i$ and we set $\dAistar = \dmax - t_i^\star$, where $t_i^\star$ is given by the optimal solution to \eqref{prog:thm2}.
	
	Furthermore, the program is always feasible.
\end{theorem}
\begin{proof}	
	Recall that the condition we need to attain to guarantee local exponential stability is that $$J_{11} = \diag\left(1 - \PhiBss \right)(\bA)^T - \diag(\dA)$$ is Hurwitz.  Noting that the only negative values of $J_{11}$ are from the term $-\diag(\dA)$, we can assert that the matrix $J_{11} + \dmax I + \epsilon I$ is a non-negative matrix, since each $\dAi \leq \dmax$ by definition.  Moreover, since $J_{11}$ is irreducible, $J_{11} + \dmax I + \epsilon I$ must be so as well.  
	
	Proposition \ref{prop:pf} then gives the existence of $\lambda > 0$ and $\vec{u}$ with $u_i > 0$ for all $i$ such that the equation $$(J_{11} + \dmax I + \epsilon I) \vec{u} = \lambda \vec{u}$$ is satisfied. If we relax the equation and make the substitution $t_i = \dmax - \dAi$ for all $i$, we can see that the inequalities:
	\begin{gather}
		\frac{\sum_{\jini{\Ain}} \bAji z_i u_j + t_i u_i + \epsilon u_i}{\lambda u_i} \leq 1 \; \quad \forall i, \label{const:spec_1}
	\end{gather}
	compose eigenvalue equations when met with equality.  It remains to show that any optimal solution to the geometric program is such that these constraints are met with equality.
	
	For purposes of identifying a contradiction, assume that there exists an optimal solution in which $\bAijstar$ is the computed optimal value of $\bAij$ for some constraint $i$ for which \eqref{const:spec_1} was not met with equality.  Noting that $\bAij$ affects no other constraint, we may increase $\bAijstar$ to some other value $\tbAij > \bAijstar$ such that \eqref{const:spec_1} is met with equality.  In doing so, we improve the value of the objective function, since $f_{ij}$ was specified as monotonically decreasing.  It must then be that our supposed solution was not optimal, and we have proven that \eqref{const:spec_1} is met with equality at any optimal solution.  By noting that the constraint $0 \leq \lambda \leq \dmax$ holds, we see that the leading eigenvalue of $J_{11}$ is negative, and the extinction condition required by Theorem \ref{thm:dyn_stable} is realized.  By applying our use of the posynomial transform, we may set $\dAistar = \dmax - t_i^\star$.
	
	It remains to prove the existence of a feasible solution for any permissible choice of program data.  We proceed by construction.  Select $\bA = \alpha A$ and $\dA = \gamma \vec{1}$.  Then, we can write the eigenvalue constraint as:
	{\small
	\begin{equation*}
		\begin{aligned}
	\lambda_{\max} \left(\diag\left(1 - \PhiBss\right) (\alpha A)^T - \gamma I + \dmax I + \epsilon I \right) < \dmax 
		\end{aligned}
	\end{equation*}
	}
	where if we choose $\gamma = \underset{i}{\min}\lbr \hat{\delta}_i^\virA \rbr$, we can reduce this to:
	$$\lambda_{\max}\left( \diag\left(1 - \PhiBss\right)(A)^T \right)< \frac{(\gamma - \epsilon)}{\alpha}$$  Since we can choose any $\alpha > 0$, our proof is complete.
\end{proof}

\begin{remark}
{\rm
	In a strict sense, \eqref{prog:thm2} generates an optimal solution to \eqref{prog:opt_extinct} in that the $\epsilon$ required by the statement of Theorem \ref{thm:dyn_stable} may be chosen arbitrarily close to $0$.  However, for any fixed $\epsilon > 0$, the solution obtained will be suboptimal.
	} \oprocend
\end{remark}

\begin{remark}[Cost Function Restrictions]
{\rm
We have solved Problem \ref{prob:min_stabilization} for the specified class of cost functions.  However, this restriction is slight within the context of the problem.  Given that the parameter $\bAij$ is a rate of spread, it is natural to associate it with a monotonically decreasing cost function; this captures the intuition that enforcing a phenomenon to be less aggressive is costly when attempting to extinct it.  Since we may choose any $g_i \in \Posy(0,\dmaxAi)$, our possible choices for $g_i$ are many.  To make the extent of this flexibility concrete, we note that $\Posy(0,\dmaxAi)$ includes the class of shifted finite-order polynomials with positive coefficients.
} \oprocend
\end{remark}



We now shift our focus to a setting in which exponential extinction may not be possible.  In particular, we consider a situation in which we are given a fixed operating budget $\budget > 0$, and we are tasked with mitigating the spread of the unwanted behavior insofar as it is possible.  As a proxy for making best use of the resources available, we are interested in solving the following problem: 
\begin{problem}[Fixed Budget Mitigation]\label{prob:rate_min}
	For some specified $S I_1 S I_2 S$ spreading processes on a bilayer graph $G$ and given cost functions, determine an allocation of resources which conforms to a budget~$\budget > 0$ and mitigates the extent of spread of a chosen behavior $\virA$ to whatever extent possible. 
\end{problem}

Our approach to this problem may be formalized as follows.  Since our budget is fixed, it may well be the case that our resources are insufficient for attaining the exponential extinction condition of Theorem \ref{thm:stab}.  However, we would like to design a program which recovers this condition whenever possible.  This may be accomplished by choosing the real component of the leading eigenvalue as the objective to our program, and adjusting the feasibility set accordingly.  We formalize this approach with the following:
\begin{theorem} \label{thm:min_rate}
	Consider an instance of the dynamics \eqref{eq:dynA}-\eqref{eq:dynB} with an equilibrium point of the form $\Phiss = \left[(\PhiAss)^T, (\PhiBss)^T \right]^T$ with $\PhiAiss = 0$ for all $i$ and $\PhiBss$ given by the solutions of \eqref{eq:equilB}.  Define $z_i = \left(1 - \PhiBiss \right)$ for all $i$.
	
	Then, for any $S I_1 S I_2 S$ spreading process on a strongly connected bilayer graph $G$, any set of monotonically decreasing posynomial cost functions $\lbr f_{ij} \rbr_{\{(i,j) \in E^A\}}$, any set of functions $\left\{g_i \in \Posy(0,\dmaxAi)\right\}_{i = 1}^{|V^A|},$ and any budget $\budget > 0$, Problem \ref{prob:rate_min} can be solved by the following geometric program:
	\begin{equation} \label{prog:thm3}
		\begin{aligned}
			& \underset{\{\bA, \vec{t}, \lambda , \vec{u}\}}{\text{minimize}}
			& & \lambda \\
			& \text{subject to}
			& & \frac{\sum_{\jini{\Ain}} \bAji z_i u_j + t_i u_i}{\lambda u_i} \leq 1 \; \forall i,\\
			&&& \frac{\sum_{\{(i,j) \in E^A\}} f_{ij} \left(\bAij \right) + \hat{g}_i \left(t_i \right)}{\budget} \leq 1 \; \forall i,\\
			&&& \frac{t_i}{\dmax} \leq 1 \; \forall i,\\
			&&& \frac{\left(\dmax - \dmaxAi \right)}{t_i} \leq 1 \; \forall i,\\
			&&& \bAij, \, u_i  \geq 0, \; \forall \; i,j \in V\\
		\end{aligned}
	\end{equation}
	where $\dmax > \underset{i}{\max}\lbr\dmaxAi\rbr$ and $\hat{g}_i$ denotes the posynomial transform of $g_i$, and we set $\dAistar = \dmax- t_i^\star$, where $t_i^\star$ is given by the optimal solution to \eqref{prog:thm3}.
\end{theorem}
\begin{proof}
	We will show that the stated geometric program is an equivalent problem to minimizing the eigenvalue of $J_{11} = \diag\left(1 - \PhiBss \right)( \bA)^T - \diag(\dA)$. This will assure that when the specified cost is above the optimal cost threshold, we recover the desired extinction condition, and we otherwise try to come as close as possible.  Noting that $J_{11}$ is irreducible by construction and that $\dmax$ functions as an upper bound for all terms $\dAi$, it must be that $J_{11} + \dmax I$ is non-negative and irreducible.  Hence, Proposition \ref{prop:pf} applies and we must have the existence of some $\vec{u}$ such that $u_i > 0$ for all $i$ such that $(J_{11} + \dmax I) \vec{u} = \lambda \vec{u}$.
	
	As in the proof for Theorem \ref{thm:dyn_stable}, we can relax the eigenvalue equations with the substitution $t_i = \dmax - \dAi$ to obtain the inequalities:
	\begin{gather}
		\frac{\sum_{\jini{\Ain}} \bAji z_i u_j + t_i u_i}{\lambda u_i} \leq 1, \; \forall i. \label{const:spec_2}
	\end{gather}
	  
	To show how we may attain equality of \eqref{const:spec_2} at an optimal solution of \eqref{prog:thm3}, we may make a similar argument as to the proof of Theorem \ref{thm:dyn_stable}.  We will show that there always exists an optimal solution which meets the constraint with equality, and provide a method for constructing it.
	
	Suppose that there exists some optimal solution $$\lbr \lbr \bAijstar\rbr_{(i,j) \in E^A}, \lbr\dAistar, u_i^\star, t_i^\star \rbr_{i = 1}^{|V^A|}, \lambda^\star \rbr$$ at which \eqref{const:spec_2} is not met with equality for some $i$.  Since the $f_{ij}$ functions are monotonically decreasing, we may increase the value of $\bAij$ for some edge $(i,j)$ until equality is attained without violating the budget constraint.  As this increase neither changes the value of $\lambda$ nor make the solution infeasible, it must be that the new solution is optimal.  Hence, given any optimal solution of \eqref{prog:thm3}, we may compute an optimal solution with equality in \eqref{const:spec_2} by increasing values of $\bAij$. 
	
	Given the optimal solution in which \eqref{const:spec_2} is met with equality, we may then set $\dAistar = \dmax - t_i^\star$ to recover the values necessary to solve Problem \ref{prob:rate_min}.
\end{proof}
\begin{remark}
	{\rm
	Note that the program is convex for any specified $\left\{g_i \in \Posy(0,\dmaxAi)\right\}_{i = 1}^{|V^A|}$, however particular choices of $g_i$ may have strictly positive minimum values.  Hence, there exists the possibility that \eqref{prog:thm3} is infeasible.  This difficulty is avoided if we restrict our choices of $g_i$ further, e.g. to functions which satisfy $\lim_{\{z \goesto 0^+\}} g_i(z) = 0$.} \oprocend
\end{remark}
\begin{remark}
	{\rm
	Note that in the event that the specified budget is not enough for extinction, the method presented here only mitigates the epidemic spread in a heuristic sense.  Formal proof that the eigenvalue minimization specified is a good proxy for optimizing the attained steady state of the unwanted behavior is unavailable, however we show in Section \ref{sec:Sim} that the approach works well in simulation.} \oprocend
\end{remark}

We close this section by noting that the optimization programs \eqref{prog:thm2} and \eqref{prog:thm3} may be specialized to particular applications by the addition of further parameter constraints.  Of particular interest may be the inclusion of box constraints, such that we have $\bAij \in [\bAijlow, \bAijup]$ for all $i$ and $j$, and $\dAi \in [\dAilow, \dAiup]$ for all $i$, which would model a scenario in which some parameter values are only partially designable.  In addition, we may adding constraints which enforce equality between various parameters in order to reflect a situation in which control of each spreading or healing rate cannot happen in isolation.  However, since these extensions occasion no further mathematical difficulties, we will not explicitly consider them here.

\section{Simulations and Discussion} \label{sec:Sim}
Our simulations accomplish two tasks.  In Section \ref{subsec:opt_sims}, we consider the performance of the optimization methods designed with respect to the intended design goals of the procedures, and find that in the mean-field regime, both methods work well.  In Section \ref{subsec:meanfield}, we consider the accuracy of the mean-field model studied.

\subsection{Optimization Simulations} \label{subsec:opt_sims}
We consider the mean-field dynamics of a $100$-node multilayer graph, with $80$ nodes in the layer spreading behavior $\virA$, $80$ nodes in the layer spreading behavior $\virB$, and $60$ nodes in the layer intersection.  Each graph layer studied is a randomly generated strongly connected digraph, with the set of overlapping nodes selected uniformly at random from each subgraph.  As cost functions, we study $f_{ij}(\bAij) = \frac{1}{\bAij}$ for each edge, and $g_i(\dAi) = (\dmaxAi - \dAi)^2 + (\dmaxAi - \dAi)$ for each node.
	
From our analysis, we expect that a solution generated from \eqref{prog:thm2} will attain the extinction condition in a ``tight,'' sense: if the contagion were made any more aggressive, we should expect it would survive.  This is exactly what happens in Figure \ref{fig:thresh}, in which we consider the results of a sensitivity analysis of the solutions generated by \eqref{prog:thm2}.  Here, we plot the attained mean-field steady states as a function of a scaling of a solution generated by \eqref{prog:thm2} by the factor $\alpha$.  It is precisely when $\alpha > 1$ that the behavior survives in an endemic state, as expected. 

We study the efficacy of the fixed budget network design in Figure \ref{fig:budget_thresh}, where we perform a similar sensitivity analysis as the one presented in Figure \ref{fig:thresh}.  We plot the average mean-field steady state values attained by a graph designed by \eqref{prog:thm3} with a budget given by $\alpha \budget^\star$, where $\budget^\star$ is the optimum value of the budget of the given problem instance, as computed by \eqref{prog:thm2}.  We see that for all values of $\alpha \geq 1$ extinction is attained, as predicted by construction.  For values of $\alpha < 1$, we see that the behavior survives in an endemic state.  Moreover, the simulations seem to indicate that the attained steady state grows continuously with decreasing budget, which suggests that the eigenvalue minimization problem serves as a suitable approach to network design when the given budget is fixed.  

{\psfrag{0}[cc][cc]{\tiny $0$}
	\psfrag{0.1}[cc][cc]{\tiny \hspace*{-2mm} 0.1}
	\psfrag{0.2}[cc][cc]{\tiny \hspace*{-2mm} 0.2}
	\psfrag{0.3}[cc][cc]{\tiny \hspace*{-2mm} 0.3}
	\psfrag{0.4}[cc][cc]{\tiny \hspace*{-2mm} 0.4}
	\psfrag{0.5}[cc][cc]{\tiny \hspace*{-2mm} 0.5}
	\psfrag{0.6}[cc][cc]{\tiny \hspace*{-2mm} 0.6}
	\psfrag{0.7}[cc][cc]{\tiny \hspace*{-2mm} 0.7}
	\psfrag{0.8}[cc][cc]{\tiny \hspace*{-2mm} 0.8}
	\psfrag{0}[cc][cc]{\tiny \hspace*{-1mm} 0}
	\psfrag{5}[cc][cc]{\tiny \hspace*{-1mm} 5}
	\psfrag{10}[cc][cc]{\tiny \hspace*{-1mm} 10}
	\psfrag{15}[cc][cc]{\tiny \hspace*{-1mm} 15}
	\psfrag{20}[cc][cc]{\tiny \hspace*{-1mm} 20}
	\psfrag{25}[cc][cc]{\tiny \hspace*{-1mm} 25}
	\psfrag{30}[cc][cc]{\tiny \hspace*{-1mm} 30}
	\psfrag{35}[cc][cc]{\tiny \hspace*{-1mm} 35}
	\psfrag{40}[cc][cc]{\tiny \hspace*{-1mm} 40}
	\psfrag{thisisthresholdlong}[cc][cc]{\tiny \hspace{.5mm}Threshold}
	\psfrag{thisisphi1}[cc][cc]{\tiny $\PhiB$}
	\psfrag{thisisphi2}[cc][cc]{\tiny $\PhiA$}
	\psfrag{0}[cc][cc]{\tiny $0$}
	\psfrag{0.1}[cc][cc]{\tiny  0.1}
	\psfrag{0.2}[cc][cc]{\tiny  0.2}
	\psfrag{0.3}[cc][cc]{\tiny  0.3}
	\psfrag{0.4}[cc][cc]{\tiny  0.4}
	\psfrag{0.5}[cc][cc]{\tiny  0.5}
	\psfrag{0.6}[cc][cc]{\tiny  0.6}
	\psfrag{0.7}[cc][cc]{\tiny  0.7}
	\psfrag{0.8}[cc][cc]{\tiny  0.8}
	\psfrag{0.9}[cc][cc]{\tiny  0.9}
	\psfrag{1}[cc][cc]{\tiny \hspace*{-2mm} 1.0}
	\psfrag{1.1}[cc][cc]{\tiny  1.1}
	\psfrag{1.2}[cc][cc]{\tiny  1.2}
	\psfrag{1.3}[cc][cc]{\tiny  1.3}
	\psfrag{1.4}[cc][cc]{\tiny  1.4}
	\psfrag{1.5}[cc][cc]{\tiny \vspace*{-1.5mm} 1.5}
	\psfrag{2}[cc][cc]{\tiny \vspace*{-1.5mm} 2}
	\vspace{-10pt}
	\begin{figure}[] 
		\begin{center}
			{\includegraphics[width = 0.5\textwidth]{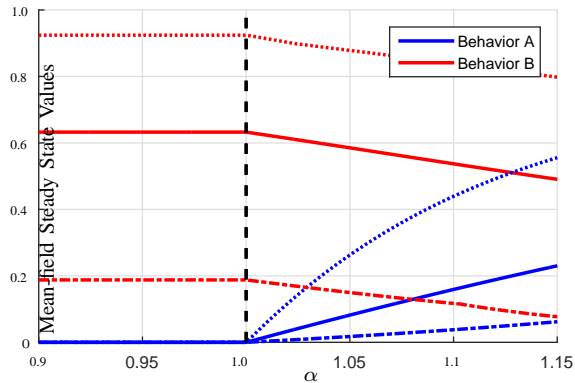}}
			\put(-120,2){\footnotesize $\alpha$}
			\put(-220,20){\footnotesize \rotatebox{90}{Mean-field Steady State Values}}	
		\end{center}
		\caption{A plot studying the sensitivity of the steady states of the heterogeneous mean-field $S I_1 S I_2 S$ model to scaling of solutions generated by \eqref{prog:opt_extinct}, denoted here by $\bAstar$.  The maximum, average, and minimum mean-field steady state values for $\bA = \alpha \bAstar$ are plotted on the $y$-axis, with the scale factor $\alpha$ is plotted on the $x$-axis.
		}  \label{fig:thresh}
	\end{figure}}

{
\psfrag{0}[cc][cc]{\tiny 0}
\psfrag{0.1}[cc][cc]{\tiny  0.1}
\psfrag{0.2}[cc][cc]{\tiny  0.2}
\psfrag{0.3}[cc][cc]{\tiny  0.3}
\psfrag{0.4}[cc][cc]{\tiny  0.4}
\psfrag{0.5}[cc][cc]{\tiny  0.5}
\psfrag{0.6}[cc][cc]{\tiny  0.6}
\psfrag{0.7}[cc][cc]{\tiny  0.7}
\psfrag{0.8}[cc][cc]{\tiny  0.8}
\psfrag{0.9}[cc][cc]{\tiny  0.9}
\psfrag{0.95}[cc][cc]{\tiny  0.95}
\psfrag{1.05}[cc][cc]{\tiny  1.05}
\psfrag{1.1}[cc][cc]{\tiny  1.10}
\psfrag{1.15}[cc][cc]{\tiny  1.15}
\psfrag{1}[cc][cc]{\tiny \hspace*{-2mm} 1.0}
\psfrag{1.5}[cc][cc]{\tiny  1.5}
\psfrag{2}[cc][cc]{\tiny  2.0}
\psfrag{2.5}[cc][cc]{\tiny  2.5}
\psfrag{4}[cc][cc]{\tiny  4.0}
\psfrag{3}[cc][cc]{\tiny  3.0}
\psfrag{3.5}[cc][cc]{\tiny  3.5}

\begin{figure}[htb] 
\begin{center}
{\includegraphics[width = 0.45\textwidth]{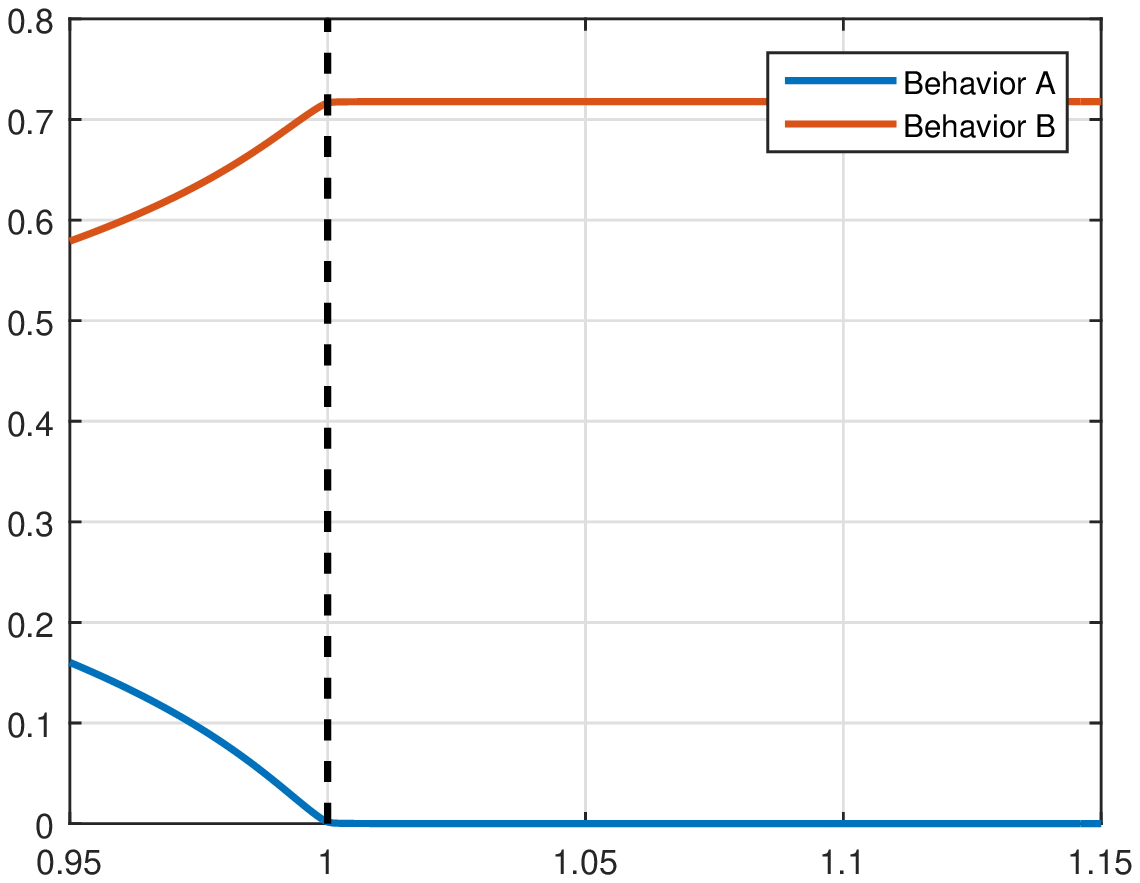}}
\put(-115,2){\footnotesize $\alpha$}
\put(-220,23){\footnotesize \rotatebox{90}{Steady State Mean-field Infection Rates}}
\end{center}
\caption{A plot of the mean steady state mean-field values of $\virA$, and $\virB$ of the solution of the optimization of Theorem \ref{thm:min_rate} with budget given by $\alpha \budget^\star$ against $\alpha$, where $\budget^\star$ is the optimal budget of Theorem \ref{thm:dyn_stable}.} \label{fig:budget_thresh}
\end{figure}}

\subsection{Mean-field Simulations} \label{subsec:meanfield}
To give an honest accounting of the utility of our results, it is necessary to investigate the relation between the behavior of the mean-field approximation we study and the $S I_1 S I_2 S$ process itself.  We first study graphs generated as solutions to \eqref{prog:thm2} with data generated as in Section \ref{subsec:opt_sims}.  A typical result of this simulation is given in Figure \ref{fig:mc_thresh_sim}.  Here we note that the transient response of the mean-field approximation is not tight, but the steady-state values appear to be accurate.  This may be an effect of the equilibrium considered in our analysis: since epidemic $\virA$ extincts, epidemic $\virB$'s dynamics eventually recover the standard $SIS$ dynamics, which is thought to be a good approximation for sufficiently large graphs \cite{VanMieghem_Virus09}. 

To asses the limits of the mean-field model's accuracy in greater generality, we consider the behavior of the model when both epidemic $\virA$ and epidemic $\virB$ survive in an endemic state in a $100$-node bi-layer random graph.  We display this in Figure \ref{fig:mf_break}, where the results were taken from a $150$-trial simulation of the $S I_1 S I_2 S$ process, and represent a typical case.  We find that the mean-field model is not necessarily accurate in this circumstance.  For the particular instance displayed in Figure \ref{fig:mf_break}, the error incurred between the ensemble average of the simulated $S I_1 S I_2 S$ process and the mean-field approximation is $\approx 0.2$.  Moreover, by considering the spread of the observed $60\%$ confidence bounds of the process, we conclude that the sample paths are not well concentrated about the their expectation, and hence there is a non-negligible amount of stochasticity which is left unaccounted for by the mean-field model.

{
\psfrag{0}[cc][cc]{\tiny $0$}
\psfrag{0.1}[cc][cc]{\tiny \hspace*{-2mm} 0.1}
\psfrag{0.2}[cc][cc]{\tiny \hspace*{-2mm} 0.2}
\psfrag{0.3}[cc][cc]{\tiny \hspace*{-2mm} 0.3}
\psfrag{0.4}[cc][cc]{\tiny \hspace*{-2mm} 0.4}
\psfrag{0.5}[cc][cc]{\tiny \hspace*{-2mm} 0.5}
\psfrag{0.6}[cc][cc]{\tiny \hspace*{-2mm} 0.6}
\psfrag{0.7}[cc][cc]{\tiny \hspace*{-2mm} 0.7}
\psfrag{0.8}[cc][cc]{\tiny \hspace*{-2mm} 0.8}
\psfrag{0}[cc][cc]{\tiny \hspace*{-1mm} 0}
\psfrag{5}[cc][cc]{\tiny \hspace*{-1mm} 5}
\psfrag{10}[cc][cc]{\tiny \hspace*{-1mm} 10}
\psfrag{15}[cc][cc]{\tiny \hspace*{-1mm} 15}
\psfrag{20}[cc][cc]{\tiny \hspace*{-1mm} 20}
\psfrag{25}[cc][cc]{\tiny \hspace*{-1mm} 25}
\psfrag{30}[cc][cc]{\tiny \hspace*{-1mm} 30}
\psfrag{35}[cc][cc]{\tiny \hspace*{-1mm} 35}
\psfrag{40}[cc][cc]{\tiny \hspace*{-1mm} 40}
\psfrag{thisisthresholdlong}[cc][cc]{\tiny \hspace{.5mm}Threshold}
\psfrag{thisisphi1}[cc][cc]{\tiny $\PhiB$}
\psfrag{thisisphi2}[cc][cc]{\tiny $\PhiA$}
\psfrag{0}[cc][cc]{\tiny $0$}
\psfrag{0.1}[cc][cc]{\tiny  0.1}
\psfrag{0.2}[cc][cc]{\tiny  0.2}
\psfrag{0.3}[cc][cc]{\tiny  0.3}
\psfrag{0.4}[cc][cc]{\tiny  0.4}
\psfrag{0.5}[cc][cc]{\tiny  0.5}
\psfrag{0.6}[cc][cc]{\tiny  0.6}
\psfrag{0.7}[cc][cc]{\tiny  0.7}
\psfrag{0.8}[cc][cc]{\tiny  0.8}
\psfrag{0.9}[cc][cc]{\tiny  0.9}
\psfrag{1}[cc][cc]{\tiny \hspace*{-2mm} 1.0}
\psfrag{1.1}[cc][cc]{\tiny  1.1}
\psfrag{1.2}[cc][cc]{\tiny  1.2}
\psfrag{1.3}[cc][cc]{\tiny  1.3}
\psfrag{1.4}[cc][cc]{\tiny  1.4}
\psfrag{1.5}[cc][cc]{\tiny \vspace*{-1.5mm} 1.5}
\psfrag{2}[cc][cc]{\tiny \vspace*{-1.5mm} 2}

\begin{figure} 
\centering
\subfigure[Epidemic $\virA$]{\includegraphics[width=.5\textwidth]{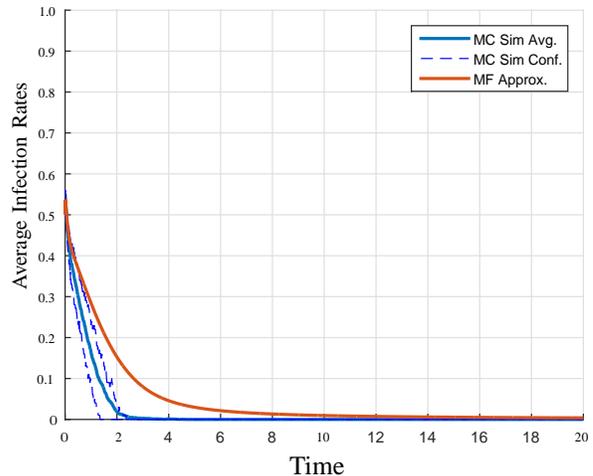}}
\put(-135,0){Time}
\put(-240,70){\footnotesize \rotatebox{90}{Average Infection Rates}}
\hfill
\subfigure[Epidemic $\virB$]{\includegraphics[width=.5\textwidth]{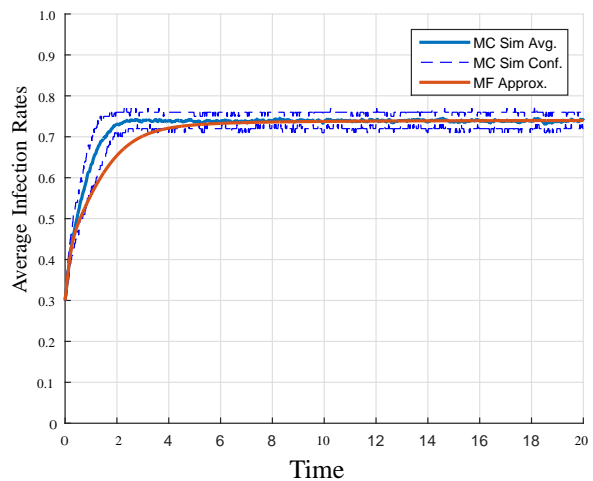}}
\put(-135,0){Time}
\put(-240,70){\footnotesize \rotatebox{90}{Average Infection Rates}}
\caption{The spreading behavior of the $S I_1 S I_2 S$ process as compared to the mean-field approximation.  These simulations were performed on the graphs used for the simulations of Section \ref{subsec:opt_sims}.  The confidence bounds plotted contain $60\%$ of the sample paths of the simulations.}\label{fig:mc_thresh_sim}
\end{figure}}

{
	\psfrag{0}[cc][cc]{\tiny $0$}
	\psfrag{0.1}[cc][cc]{\tiny \hspace*{-2mm} 0.1}
	\psfrag{0.2}[cc][cc]{\tiny \hspace*{-2mm} 0.2}
	\psfrag{0.3}[cc][cc]{\tiny \hspace*{-2mm} 0.3}
	\psfrag{0.4}[cc][cc]{\tiny \hspace*{-2mm} 0.4}
	\psfrag{0.5}[cc][cc]{\tiny \hspace*{-2mm} 0.5}
	\psfrag{0.6}[cc][cc]{\tiny \hspace*{-2mm} 0.6}
	\psfrag{0.7}[cc][cc]{\tiny \hspace*{-2mm} 0.7}
	\psfrag{0.8}[cc][cc]{\tiny \hspace*{-2mm} 0.8}
	\psfrag{0}[cc][cc]{\tiny \hspace*{-1mm} 0}
	\psfrag{5}[cc][cc]{\tiny \hspace*{-1mm} 5}
	\psfrag{10}[cc][cc]{\tiny \hspace*{-1mm} 10}
	\psfrag{15}[cc][cc]{\tiny \hspace*{-1mm} 15}
	\psfrag{20}[cc][cc]{\tiny \hspace*{-1mm} 20}
	\psfrag{25}[cc][cc]{\tiny \hspace*{-1mm} 25}
	\psfrag{30}[cc][cc]{\tiny \hspace*{-1mm} 30}
	\psfrag{35}[cc][cc]{\tiny \hspace*{-1mm} 35}
	\psfrag{40}[cc][cc]{\tiny \hspace*{-1mm} 40}
	\psfrag{thisisthresholdlong}[cc][cc]{\tiny \hspace{.5mm}Threshold}
	\psfrag{thisisphi1}[cc][cc]{\tiny $\PhiB$}
	\psfrag{thisisphi2}[cc][cc]{\tiny $\PhiA$}
	\psfrag{0}[cc][cc]{\tiny $0$}
	\psfrag{0.1}[cc][cc]{\tiny  0.1}
	\psfrag{0.2}[cc][cc]{\tiny  0.2}
	\psfrag{0.3}[cc][cc]{\tiny  0.3}
	\psfrag{0.4}[cc][cc]{\tiny  0.4}
	\psfrag{0.5}[cc][cc]{\tiny  0.5}
	\psfrag{0.6}[cc][cc]{\tiny  0.6}
	\psfrag{0.7}[cc][cc]{\tiny  0.7}
	\psfrag{0.8}[cc][cc]{\tiny  0.8}
	\psfrag{0.9}[cc][cc]{\tiny  0.9}
	\psfrag{1}[cc][cc]{\tiny \hspace*{-2mm} 1.0}
	\psfrag{1.1}[cc][cc]{\tiny  1.1}
	\psfrag{1.2}[cc][cc]{\tiny  1.2}
	\psfrag{1.3}[cc][cc]{\tiny  1.3}
	\psfrag{1.4}[cc][cc]{\tiny  1.4}
	\psfrag{1.5}[cc][cc]{\tiny \vspace*{-1.5mm} 1.5}
	\psfrag{2}[cc][cc]{\tiny \vspace*{-1.5mm} 2}
\begin{figure} 
	\centering
	\subfigure[Behavior $\virA$]{\includegraphics[width=.5\textwidth]{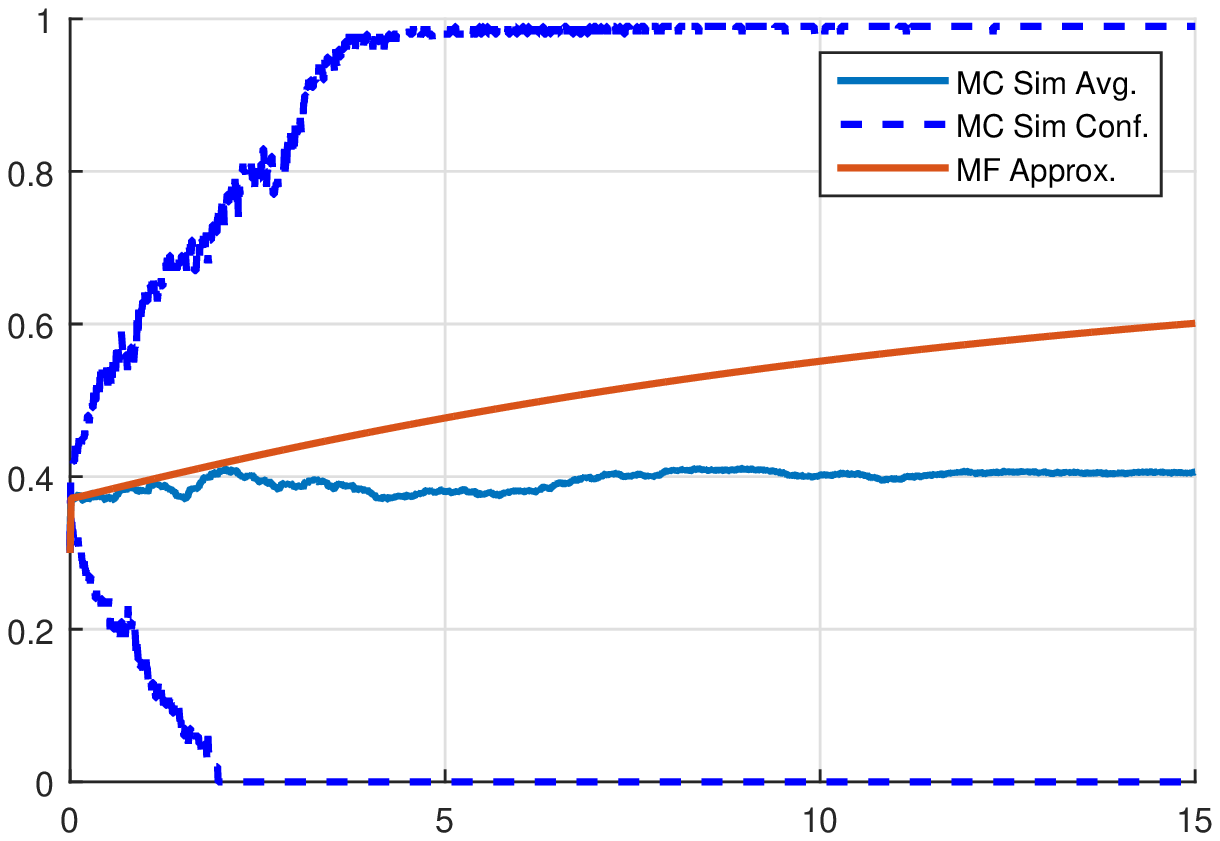}}
	\put(-135,0){Time}
	\put(-240,50){\footnotesize \rotatebox{90}{Average Infection Rates}}
	\hfill
	\subfigure[Behavior $\virB$]{\includegraphics[width=.5\textwidth]{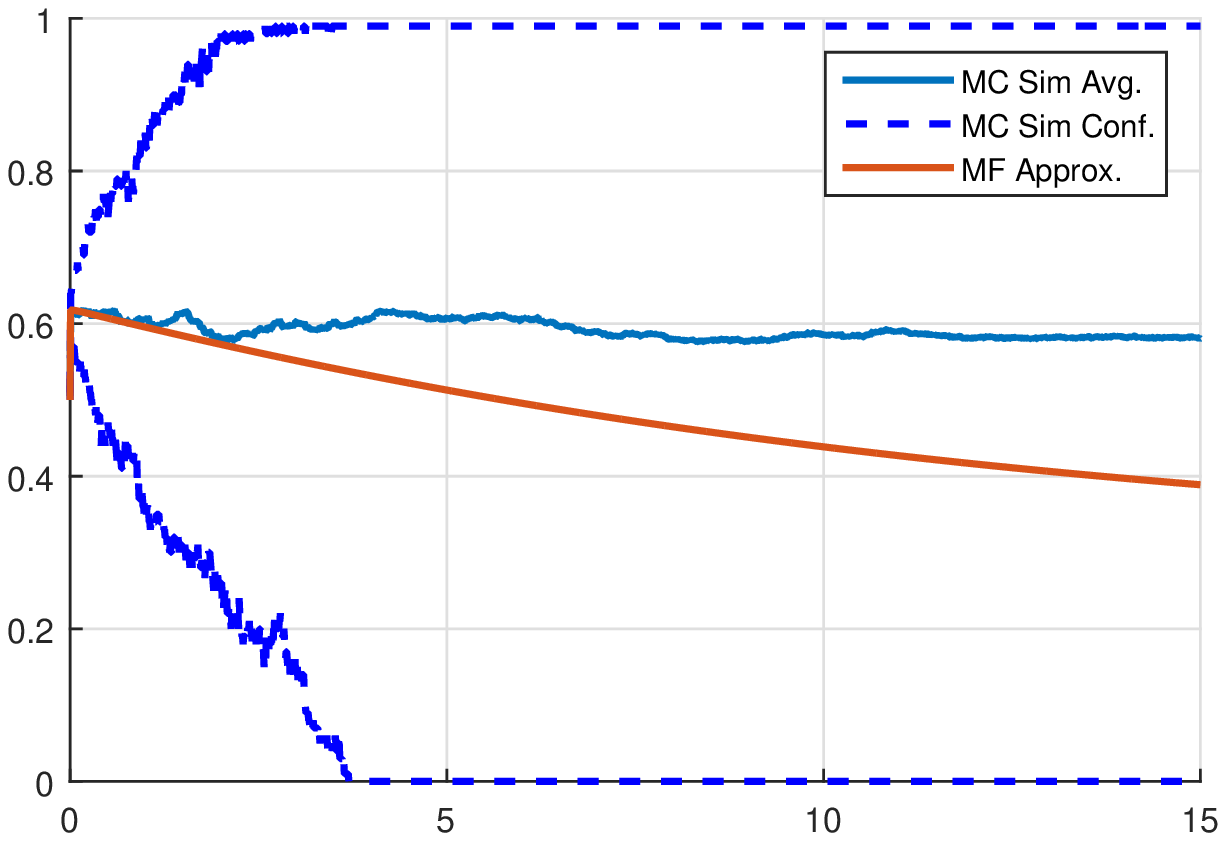}}
	\put(-135,0){Time}
	\put(-240,50){\footnotesize \rotatebox{90}{Average Infection Rates}}
	\caption{The spreading behavior of the $S I_1 S I_2 S$ process as compared to the mean-field approximation.  These simulations were performed on $100$-node random graphs with $50\%$ sparsity.  The confidence bounds plotted contain $60\%$ of the sample paths of the simulations.}\label{fig:mf_break}
\end{figure}}

\section{Summary and Future Work} \label{sec:Discuss}

The class of multilayer spreading processes is one with much potential.  We have managed to define a framework in which the earlier work on competitive multilayer processes can be extended to a class of heterogeneously parametrized processes on generalized graph layers.  Moreover, we have provided a first step in analyzing competitive multilayer spreading processes by finding necessary and sufficient conditions for the exponential stability for any equilibrium of the system in which one process extincts exponentially quickly and the other survives in an endemic state.  

Furthermore, we have developed an optimization program for determining optimal-cost parameter distributions such that the desired equilibrium is stabilized, and another which performs a heuristic design in the case of a fixed budget.  We have found that the designed optimization routines work well in simulation in both cases.  The marketing problem we posed as motivation is just one example of many which can be posed for a class of such equilibria.  By redefining the meaning of the variable states, we can apply our model to diverse settings; potential examples include optimizing political strategies and protecting against viral spread.  

This work opens many possible avenues for future research. A useful generalization would be an extension to a k-layer, k-process framework, as such an extension could greatly improve the modeling capacity of the tools developed.  Additionally, we can place further assumptions on the set of controllable parameters and the objective of our resources allocations.  For example, it may be reasonable to have control over both the spreading parameters of $\virA$ and $\virB$, in which case it may be desirable to specify a steady state and compute an optimal allocation which attains it.  It may also be of interest to further define tractable methods for computing conditions of coexistence and extinction - both of which have a useful interpretation in certain contexts.

Perhaps the most important question left open pertaining to this work is the relation of the mean-field approximation to the underlying stochastic process.  While some form of approximation is necessary in order to avoid the exponential state space of the exact representation of the system, it is clear that other approximation schemes can be considered, and it is currently unclear which methods are most effective in which contexts.  While the approximation technique applied here works well for the extinction problem considered in simulation, it is clear that a more precise understanding of the interrelation between the approximated dynamics and the exact dynamics is a substantial requirement for future work.
\section*{Acknowledgements}
This work was supported in part by the National Science Foundation grant CNS-1302222 “NeTS: Medium: Collaborative Research: Optimal Communication for Faster Sensor Network Coordination”, IIS-1447470 “BIGDATA: Spectral Analysis and Control of Evolving Large Scale Networks”, and the TerraSwarm Research Center, one of six centers supported by the STARnet phase of the Focus Center Research Program (FCRP), a Semiconductor Research Corporation program sponsored by MARCO and DARPA.
\vspace{-7pt}
\bibliography{master}
\vspace{-15pt}
\bibliographystyle{IEEEtran}
\vspace{10pt}

\appendix

\begin{proof}[Lemma 1]
	By using the variational characterization of eigenvalues, it will suffice to show that 
	\begin{equation} \label{eq:eign_ineq2}
		\begin{aligned}
		&\sup_{\vec{v} \neq 0}\Re \left[ \frac{(\vec{v})^* \left( \diag(\vec{\alpha})M - \diag(\vec{\gamma})\right)\vec{v}}{(\vec{v})^* \vec{v}} \right] \\
		&\leq \sup_{\vec{v} \neq 0}\Re \left[ \frac{(\vec{v})^* \left(M - \diag(\vec{\gamma})\right)\vec{v}}{(\vec{v})^* \vec{v}} \right]
		\end{aligned}
	\end{equation} holds, where we use $\Re$ to denote an operator which returns the real part of its argument and $(\cdot)^*$ denotes the conjugate transpose of a vector.  Our demonstration of this fact requires that we establish two pieces: (i) we may evaluate the supremums over $\vec{v} \in \realnonnegative^n$ without affecting their attained value, and (ii) for each $\vec{v} \in  \realnonnegative^n$, the desired inequality follows immediately.
	To argue (i), we will only explicitly consider the left hand side of \eqref{eq:eign_ineq2}, the right hand side follows from similar arguments.  Fix some $\tilde{v} \in \complex^n$ with components $\tilde{v}_r = \tilde{x}_r + i \tilde{y}_r$; we will show that we may always construct some vector $\hat{v} \in \realnonnegative^n$ which increases the value of the function evaluated by the supremum.  For our choice of $\tilde{v}$, we may write the argument of the the left hand side of \eqref{eq:eign_ineq2} as
	\begin{equation*}
		\Re \left[\frac{\sum_{k \neq \ell} (\tilde{v}_k)^* \tilde{v}_{\ell} \alpha_k m_{k \ell} + \sum_{k} (\tilde{v}_k)^* \tilde{v}_k \gamma_k}{\tilde{v}^* \tilde{v}} \right].
	\end{equation*}
	Now, consider the vector $\hat{v}$ defined as $\hat{v}_r = \sqrt{x_r^2 + y_r^2}$ for all $r$.  By construction, we have $(\tilde{v}_k)^* \tilde{v}_k = \hat{v}_k^2$ holds for all $k$, so it will suffice to show that 
	\begin{equation*}
		\begin{aligned}
			&\Re \left[\sum_{k \neq \ell} \tilde{v}_k^* \tilde{v}_{\ell} \alpha_k m_{k \ell} \right] \leq \Re \left[\sum_{k \neq \ell}^n \hat{v}_k \hat{v}_\ell \alpha_k m_{k \ell}\right] 
		\end{aligned}
	\end{equation*}
	holds.  This follows immediately once we establish that the inequality $\Re((\tilde{v}_k)^* \tilde{v}_\ell) \leq  \hat{v}_k \hat{v}_\ell$ holds for all choices of $k$ and $\ell$.  This can be verified by direct computation of the corresponding inequality for the squares of these values:
	\begin{equation*}
		\begin{aligned}
			\left[\Re(\tilde{v}_k^* \tilde{v}_\ell) \right]^2 &= x_k^2 x_\ell^2 + y_k^2 y_\ell^2 + 2 x_k x_\ell y_k y_\ell \\
			&\leq x_k^2 x_\ell^2 + y_k^2 y_\ell^2 + x_k^2 y_\ell^2 + x_\ell^2 y_k^2\\
			&= \left(\left(x_k^2 + y_k^2\right)^{\frac{1}{2}} \left(x_\ell^2 + y_\ell^2\right)^{\frac{1}{2}} \right)^2\\
			&= \hat{v}_k^2 \hat{v}_\ell^2
		\end{aligned}
	\end{equation*}
	where the inequality follows from noting that $$y_k^2 x_\ell^2 + x_k^2 y_\ell^2 -2 x_k x_\ell y_k y_\ell = (x_k y_l - x_\ell y_k)^2 \geq 0.$$
	It remains to show that
	\begin{equation*}
		\begin{aligned}
		\begin{aligned}
				&\sup_{\vec{v} \in \realnonnegative^n}\Re \left[ \frac{(\vec{v})^T \left( \diag(\vec{\alpha})M - \diag(\vec{\gamma})\right)\vec{v}}{(\vec{v})^T v} \right] \\
				&\leq \sup_{\vec{v} \in \realnonnegative^n}\Re\left[ \frac{(\vec{v})^T \left(M - \diag(\vec{\gamma})\right)\vec{v}}{(\vec{v})^T \vec{v}} \right]
				\end{aligned}
		\end{aligned}
	\end{equation*}
	holds.  We prove this by fixing any $\vec{v} \in \realnonnegative^n$, and noting that 
	\begin{equation*}
		\begin{aligned}
		&\frac{\sum_{k \neq \ell} v_k v_\ell \alpha_k m_{k \ell} - \sum_{k} v_k^2 \gamma_k}{(\vec{v})^T \vec{v}} \\
		&\leq \frac{\sum_{k \neq \ell} v_k v_\ell m_{k \ell} - \sum_{k} v_k^2 \gamma_k}{(\vec{v})^T \vec{v}}
		\end{aligned}
	\end{equation*}
	follows immediately as a consequence of $\alpha_k \in [0,1]$ for all $k$ and $m_{k \ell} \geq 0$ for all $k$ and $\ell$.
\end{proof}

\end{document}